\documentclass{aimeta2017}

\usepackage{amsmath,amssymb,amsthm}
\usepackage{mathtools}
\mathtoolsset{showonlyrefs}

\newcommand{\cB}{\mathcal{B}}
\newcommand{\cC}{\mathcal{C}}
\newcommand{\cD}{\mathcal{D}}
\newcommand{\cE}{\mathcal{E}}
\newcommand{\cF}{\mathcal{F}}
\newcommand{\cX}{\mathcal{X}}
\newcommand{\cZ}{\mathcal{Z}}

\newcommand{\N}{\mathbb{N}}
\newcommand{\R}[1]{\mathbb{R}^{#1}}

\newcommand{\de}{\mathrm{d}}
\newcommand{\diam}{\operatorname{diam}}
\newcommand{\dist}{\operatorname{dist}}
\def \oea {\Omega_\eps(a)}
\def \euno {\eps_1}
\def \edue {\eps_2}

\newcommand{\eps}{\varepsilon}
\newcommand{\rhobar}{\bar\rho}
\newcommand{\Om}{\Omega}

\renewcommand{\div}{\operatorname{div}}
\newcommand{\curl}{\operatorname{curl}}

\newtheorem{defin}{Definition}[section]
\newtheorem{theorem}[defin]{Theorem}
\newtheorem{corollary}[defin]{Corollary}

\title{QUALITATIVE AND QUANTITATIVE PROPERTIES OF THE DYNAMICS OF SCREW DISLOCATIONS}

\author{Marco Morandotti$^1$}

\heading{Marco Morandotti}

\address{$^1$Technische Universit\"at M\"unchen \\
  Boltzmannstrasse, 3, 85748 Garching b.\@ M\"unchen, Germany\\
  e-mail: marco.morandotti@ma.tum.de}


\abstract{This note collects some results on the behaviour of screw dislocation in an elastic medium.
By using a semi-discrete model, we are able to investigate two specific aspects of the dynamics, namely (i) the interaction with free boundaries and collision events and (ii) the confinement inside the domain when a suitable Dirichlet-type boundary condition is imposed.

In the first case, we analytically prove that free boundaries attract dislocations and we provide an expression for the Peach--Koehler force on a dislocation near the boundary. 
Moreover, we use this to prove an upper bound on the collision time of a dislocation with the boundary, provided certain geometric conditions are satisfied.
An upper bound on the collision time for two dislocations with opposite Burgers vectors hitting each other is also obtained.

In the second case, we turn to domains whose boundaries are subject to an external stress.
In this situation, we prove that dislocations find it energetically favourable to stay confined inside the material instead of getting closer to the boundary.
The result first proved for a single dislocation in the material is extended to a system of many dislocations, for which the analysis requires the careful treatments of the interaction terms.
}

\begin{document}

\section{INTRODUCTION}

When undergoing stress and deformations, crystalline materials can exhibit a breakdown of the lattice in which their atoms are arranged.
Among such failures of the perfectly ordered structure, dislocations are particularly relevant because their presence can significantly alter the mechanical and the physical and chemical properties of a material.
In particular, they are considered to be responsible for the plastic behaviour of materials.
Starting from the pioneering work by Volterra \cite{volterra}, the mechanism for plasticity due to dislocations was suggested simultaneously and independently by Taylor \cite{taylor}, Orowan \cite{orowan}, and Polanyi \cite{polanyi} in 1934 (see also \cite{GLP} for a rigorous proof), but it was not until 1956 that dislocations were observed experimentally \cite{HHW}.
In more recent times, the theory of dislocations has captured the attention of scientists, bridging together the applied community of physicists and engineers and the theoretical one of mathematicians.
For general treaties on dislocations, we refer the reader to \cite{nabarro,HL,HB}.

Dislocations are line defects in the crystalline structure of solids and the lattice mismatch they generate is measured by the so-called \emph{Burgers vector}.
According to whether the Burgers vector is perpendicular or parallel to the dislocation line, dislocations are called \emph{edge} or \emph{screw}, respectively.
We will restrict our attention to screw dislocations only, in a material undergoing an antiplane shear deformation.
In this case, the idealised setting is an infinite cylinder $\Omega\times\mathbb{R}$, the cross--section $\Omega$ being a subset of $\mathbb{R}^2$, in which the dislocations are lines parallel to the vertical axis, meeting $\Omega$ at a discrete set of points $\cZ=\{z_1,\ldots,z_n\}$.
Due to the modelling assumption of antiplane deformation, the Burgers vectors are all parallel to the vertical axis as well and can be identified by the \emph{Burgers moduli} $\cB=\{b_1,\ldots,b_n\}$.
It is not restrictive to assume (see \cite{HB}) that $b_i\in\{\pm1\}$ for all $i=1,\ldots,n$.

The contribution contained in the present note relies on the celebrated model for screw dislocations undergoing antiplane shear proposed by Cermelli and Gurtin in 1999 \cite{CG}.
Cermelli and Gurtin's is generally referred to as a \emph{semi-discrete} model, since it provides a link between the elastic behaviour of the body, occurring at the mesoscopic scale, and the dynamics of each dislocation, occurring at the atomistic, microscopic scale.
The link between the two scales is provided by the law of motions of the dislocations, according to which the Peach--Koehler force acting on each dislocation can be obtained (as it will be more clear later) from the renormalised energy of the whole system.

A strain field $h$ associated to the system of screw dislocations $\cZ$ with Burgers vectors $\cB e_3$ satisfies the system
\begin{equation}\label{100}
\begin{cases}
\div h=0 & \text{in $\Omega$} \\
\curl h=\sum_{i=1}^n b_i\delta_{z_i} & \text{in $\Omega$}
\end{cases}
\end{equation}
in the sense of distributions.
Here, $h\colon\Omega\to\R2$ is the strain field of the deformed body, $\delta_z$ is the Dirac delta function at the point $z\in\Omega$, and $\curl h=\partial h_2/\partial x_1-\partial h_1/\partial x_2$ is the scalar curl of the field $h$.
As a consequence of the second equation in \eqref{100}, the Burgers moduli $b_i$ can be obtained by integrating the tangential component of $h$ around a closed, simple loop $\gamma_i$ containing the dislocation $z_i$ only, namely
\begin{equation}\label{102}
b_i=\int_{\gamma_i} h\cdot t\,\de s,
\end{equation}
where $t$ is the tangent vector, and $\de s$ is the line element.
The fact that $h$ is singular at the dislocations positions is expressed by the curl equation in \eqref{100}, from which we see that $h$ fails to be the gradient of a deformation (a more regular function).
If that were the case, we would have that there exists a function $u\in H^1(\Omega)$ such that $h=\nabla u$, so that the first equation in \eqref{100} would be the more familiar equilibrium equation for elasticity, $\Delta u=0$ (considering the elastic tensor to be the identity).
The elastic energy associated with the material is given by
\begin{equation}\label{101}
\cE(h):=\frac12\int_\Omega |h(x)|^2\,\de x.
\end{equation}
It is easy to see that the energy $\cE(h)$ is not finite; indeed, consider $n=1$ and denote by $B_\eps(z)$ the open ball of radius $\eps>0$ centred at $z$.
We have that 
\begin{equation}\label{103}
\cE(h)=\lim_{\eps\to0} \frac12\int_{\Omega\setminus\overline B_\eps(z)} |h(x)|^2\,\de x=\lim_{\eps\to0} \pi|\log\eps|+O(1)=+\infty.
\end{equation}
It is usual to resort to the so-called \emph{core radius approach} to tackle the problem: cores $B_\eps(z_i)$ of radius $\eps$ are removed around each dislocation $z_i$ and the energy minimisation problem is studied in the perforated domain $\Omega_\eps:=\Omega\setminus\big(\cup_{i=1}^n \overline B_\eps(z_i)\big)$, and then a suitable limit as $\eps\to0$ is taken.
This allows to single out the logarithmic contribution to the elastic energy and to introduce the \emph{renormalised energy} $\cE_n\colon\Omega^n\to\R{}\cup\{+\infty\}$, which only depends on the dislocations positions.
Denoting by $\cE_\eps$ the energy \eqref{101} defined on the perforated domain $\Omega_\eps$, we have (see \cite{BFLM}) that at a minimiser $h_\eps$ the energy can be written as 
\begin{equation}\label{104}
\cE_\eps(h_\eps)=\frac12\int_{\Omega_\eps} |h_\eps(x)|^2\,\de x=C|\log\eps|+\cE_n(z_1,\ldots,z_n)+O(\eps),
\end{equation}
where $C>0$ is a constant that depends on the Burgers vectors.

Analogously to the theory of Ginzburg-Landau vortices \cite{BBH,SS}, the renormalised energy $\cE_n$ is the key quantity to study the dynamics of screw dislocations as well as to obtain information about the positions of the dislocations.
In the following, we will give two different expressions of $\cE_n(z_1,\ldots,z_n)$, according to the type of problem we are going to treat.

This note is organised as follows.
In Section \ref{tom} we will present the results contained in \cite{HM}.
By writing the renormalised energy $\cE_n(z_1,\ldots,z_n)$ of \eqref{104} appropriately in terms of the Robin's function for the laplacian, we obtain
\begin{itemize}
\item[(a)] a precise estimate on the Peach--Koehler force on a dislocation near a free boundary;
\item[(b)] estimates on the collision time of a dislocation near the boundary;
\item[(c)] estimates on the collision time of two dislocations with opposite Burgers moduli.
\end{itemize}
The result in (a) proves analytically that free boundaries attract dislocations; the estimates in (b) and (c) are sharp, as it is proved by some simple case experiments, see \cite[Section 4]{HM}.
To achieve (a), estimates on the gradient of the Green's function for the Laplacian are needed; (a) will in turn be used to prove (b).
The same strategy to prove (b) is used to obtain (c).
The details can be found in \cite{HM}.

In Section \ref{lmsz} we will present the results contained in \cite{LMSZ}: by imposing a suitable boundary condition of Dirichlet type and by indicating with $\cE_\eps(a)$ the position-dependent minimiser of $\cE_\eps$ among a certain class of admissible deformations, we show that
\begin{itemize}
\item[(d)] the minimiser of the position-dependent renormalised energy $\cE_1(a)$ of (the right-hand side of) \eqref{104} for one single dislocation is in the interior of the domain;
\end{itemize}
and that, considering $n$ dislocations $a_1,\ldots,a_n\in\Omega$ with Burgers moduli of the same sign,
\begin{itemize}
\item[(e)] the minimiser of the position-dependent renormalised energy $\cE_n(a_1,\ldots,a_n)$ of (the right-hand side of) \eqref{104} for $n$ dislocations is a configuration of dislocations all distinct and all contained in the interior of the domain.
\end{itemize}
Result (d) is stated in terms of convergence of the functionals $\cF_\eps\colon\overline\Omega\to\R{}\cup\{+\infty\}$ defined by $\cF(a):=\cE_\eps(a)-\pi|\log\eps|$ to a functional $\cF\colon\overline\Omega\to\R{}\cup\{+\infty\}$ which attains its minimum in the interior of $\Omega$; analogously, (e) is obtained by showing that the functionals $\cF_\eps\colon\overline\Omega{}^n\to\R{}\cup\{+\infty\}$ defined by $\cF(a_1,\ldots,a_n):=\cE_\eps(a_1,\ldots,a_n)-\pi n|\log\eps|$ converge to a functional $\cF\colon\overline\Omega{}^n\to\R{}\cup\{+\infty\}$ which attains its minimum in the interior of $\Omega$.

The appropriate assumptions on the domain are stated at the beginning of each section and are assumed to hold throughout the same section.

Section \ref{conclusions} collects some final comments and remarks, including possible future directions.
The results presented here have been obtained in collaboration with T.\@ Hudson \cite{HM} and with I.\@ Lucardesi, R.\@ Scala, and D.\@ Zucco \cite{LMSZ}.

\section{BEHAVIOUR NEAR THE BOUNDARY AND ESTIMATES ON COLLISION TIMES}\label{tom}
In this section, we assume that the cross--section $\Omega\subset\R2$ is a connected open set with $C^2$ boundary.
This regularity assumption implies that the boundary satisfies \emph{uniform interior and exterior disk conditions}: there exists $\rhobar>0$ such that for any point $x\in\partial\Omega$, there exist unique points $x_\text{int}$ and $x_\text{ext}$ such that
\begin{equation}\label{1011}
  B_{\rhobar}(x_\text{int})\subseteq \Om,\quad
  B_{\rhobar}(x_\text{ext})\subset \Om^c\quad\text{and}\quad
  \partial B_{\rhobar}(x_\text{int})\cap\partial\Om\cap
  \partial B_{\rhobar}(x_\text{ext}) = \{x\},
\end{equation}
where the superscript $c$ denotes the complement of a set. 
We shall fix such $\rhobar>0$ once and for all.
It also follows that the curvature of the boundary, $\kappa$, lies in $C(\partial\Om)$, and $\|\kappa\|_{\infty}\leq\rhobar^{-1}$. 

The Green's function of the Laplacian with Dirichlet boundary conditions on $\partial\Om$ is the (distributional) solution to
\begin{equation}\label{eq:G_Om_problem}
\begin{cases}
  -\Delta_x G_\Om(x,y) = \delta_y(x) & \text{in }\Om, \\
  G_\Om(x,y) = 0 & \text{on }\partial\Om.
\end{cases}
\end{equation}
It is a classical result that $G_\Om$ is smooth in the variable $x$ on the set $\Om\setminus\{y\}$ for any
given $y\in\Om$; is symmetric, i.e. $G_\Om(x,y) = G_\Om(y,x)$; and 
\begin{equation}\label{eq:Green}
  G_\Om(x,y) = -\frac1{2\pi}\log|x-y|+ k_\Om(x,y),
\end{equation}
where $k_\Om(x,y)$ is smooth in both arguments on $\Om$, is symmetric, and satisfies the elliptic boundary value problem
\begin{equation}\label{eq:k_Om_problem}
\begin{cases}
  -\Delta_x k_\Om(x,y) = 0 & \text{in }\Om, \\
  k_\Om(x,y) = \frac1{2\pi}\log|x-y| &\text{on }\partial\Om;
\end{cases}
\end{equation}
proofs of all of the above assertions may be found in \cite[Chapter~4]{H14}.
In addition, we also define 
\begin{equation}\label{eq:h_Om}
h_\Om(x):=k_\Om(x,x),
\end{equation}
which will turn out to be a convenient function with which to express the renormalised energy.
It can be shown that $h_\Om$ satisfies the elliptic problem (see \cite{CF85})
\begin{equation}
  -\Delta_x h_\Om(x) = \frac2\pi \mathrm{e}^{-4\pi h_\Om(x)}\quad
  \text{for all }x\in\Om.\label{eq:h_Om_problem}
\end{equation}

Using the explicit expression for the Green's function \eqref{eq:Green} and the functions defined in \eqref{eq:k_Om_problem} and \eqref{eq:h_Om}, the renormalised energy $\cE_n$ from the right-hand side of \eqref{104} of $n$ dislocations with positions $z_1,\ldots,z_n\in\Om$ and Burgers moduli $b_1,\ldots,b_n\in\{-1,+1\}$ (see, \emph{e.g.}, \cite{CG,ADLGP,BM}) may be expressed as
\begin{equation}\label{renen}
  \cE_n(z_1,\ldots,z_n)= 
  \sum_{i<j}b_ib_j\Big(k_\Omega(z_i,z_j)-\frac1{2\pi}\log|z_i-z_j|\Big) + \frac12 \sum_{i=1}^n b_i^2 h_\Omega(z_i),
\end{equation}
where the contributions of the two--body interaction terms and the one--body `self--interaction' term are highlighted.
To be more precise, 
\begin{itemize}
\item each term $h_\Om(z_i)$ is the contribution to the energy given by a dislocation sitting at $z_i$; 
\item the logarithmic terms $\log|z_i-z_j|$ account for the interaction energy of the two dislocations sitting at $z_i$ and $z_j$; 
\item the term $k_\Om(z_i,z_j)$ accounts for the interaction of the dislocation sitting at $z_i$ with the boundary response due to the dislocation sitting at $z_j$.
\end{itemize}
It is worth noticing that the interaction terms also involve the product $b_ib_j$ of the Burgers
moduli of the dislocations in a fashion similar to that of electric charges: $b_i=b_j=\pm1$ gives a positive
contribution to the energy, and tends to push two dislocations with the same sign far away from each other.
dislocations.
Finally, observe that the terms with the subscript $\Om$ depend in a crucial way on the geometry of the domain
and carry information about the interaction with the boundary.
To be thorough, the energy $\cE_n$ in \eqref{renen} should also depend on the Burgers moduli $b_1,\ldots,b_n$, but we assume these are attached to the dislocations and do not vary in time. 

The force on a dislocation, the so-called Peach--Koehler force \cite{HL}, is obtained by taking the negative of the gradient of the renormalised energy with respect to the dislocation position
\begin{equation}\label{502}
f_i(z_1,\ldots,z_n)=-\nabla_{z_i}\cE_n(z_1,\ldots,z_n),\qquad\text{for $i=1,\ldots,n$.}
\end{equation}
The subscript $i$ refers to the force experienced by the dislocation at $z_i$ and the dependence on the whole configuration of dislocations $z_1,\ldots,z_n$ highlights its nonlocal character. 

The law describing the dynamics of the dislocations is therefore expressed as
\begin{equation}\label{500}
\dot z_i(t)=-\nabla_{z_i}\cE_n(z_1(t),\ldots,z_n(t)),\qquad\text{for $i=1,\ldots,n$,}
\end{equation}
complemented with suitable initial condition at time $t=0$.

Let $d_n\colon\Omega^n\to[0,+\infty)$ be defined by 
\begin{equation}\label{1001}
  d_n(x_1,\ldots,x_n):=
  \begin{cases}
    \dist(x_1,\partial\Om) & n=1,\\
    \min_i\dist(x_i,\partial\Om)\wedge \min_{i\neq j}|x_i-x_j| &\text{otherwise}.
  \end{cases}
\end{equation}
In the case $n=1$, the function $d_1$ measures the distance of the dislocation from the boundary $\partial\Omega$, whereas, if $n\geq2$, $d_n$ describes the minimal separation among the dislocations and their distance from the boundary.

The situation we consider for addressing problems (a) and (b) of the introduction is the following: we suppose that we have $n\in\N$ dislocations in $\Omega$ one of which, $z_1$, is much closer to the boundary $\partial\Omega$ than the others; we also suppose that the other $n-1$ dislocations, $z_2,\ldots,z_n$, are spaced sufficiently far apart from each other and from the boundary.
We introduce the notation $z':=(z_2,\ldots,z_n)$ so that
the configuration of the $n$ dislocations can be represented by the vector $z:=(z_1,z')\in\Omega^n$.
Given $0<\delta<\gamma<\diam\Omega/2$, define the set
\begin{equation}\label{eq:Geomofz}
  \cD_{n,\delta,\gamma}:=\{(z_1,z')\in\Om^n : d_1(z_1)<\delta,d_{n-1}(z')>\gamma\}.
\end{equation}
The geometric meaning of the set $ \cD_{n,\delta,\gamma}$ defined above is the following: if $z\in \cD_{n,\delta,\gamma}$, it means that $z_1$ lies at a distance of at most $\delta$ from the boundary, while all the other dislocations $z_2,\ldots,z_n$ lie at a distance of at least $\gamma$ away from the boundary and their mutual distance is also at least $\gamma$.
The condition $\delta<\gamma$ ensured that $z_1$ is closer to the boundary than any other dislocation.

In the following theorem we show that the Peach--Koehler force acting on a dislocation which is very close to the boundary is directed along the outward unit normal at the boundary point closest to the dislocation.
\begin{theorem}[{(a) -- free boundaries attract dislocations \cite{HM}}]\label{thm:fatal}
Let $n\in\N$, let $\sigma\in(0,1)$, recall the definition of $\rhobar$ from \eqref{1011}, and let $\delta\in(0,\sigma\rhobar)$ and $\gamma\in(\max\{2\delta,\rhobar\},\diam\Omega/2)$.
Let $z=(z_1,z')\in\cD_{n,\delta,\gamma}$.
Then, if $s\in\partial\Omega$ is the boundary point closest to $z_1$, the Peach-Koehler force $f_1(z)$ on the dislocation $z_1$ (see \eqref{502}) satisfies
\begin{equation}\label{505}
f_1(z)=\frac{\nu(s)}{4\pi d_1(z_1)}+\frac{C_{n,\sigma}(\gamma)}{2\pi\rhobar},
\end{equation}
where $\nu(s)$ is the outward unit normal to $\partial\Omega$ at $s$ and the constant $C_{n,\sigma}(\gamma)>0$ only depends on the geometric parameter $\rhobar$, on $\sigma\in(0,1)$, and on how far all the other dislocations are from $z_1$ and from $\partial\Omega$.
\end{theorem}

We want to find conditions on the parameters $\delta$ and $\gamma$ in \eqref{eq:Geomofz}, in order to strengthen the constraint $\delta<\gamma$ in such a way that if the initial configuration of the system $z(0)\in\cD_{n,\delta_0,\gamma_0}$, for some $\delta_0<\gamma_0$, then $z_1$ will collide with the boundary before any other collision event occurs.

\begin{theorem}[{(b) -- collision with the boundary \cite{HM}}]\label{thm:boundary_collision}
Let $n\in\N$, let $\sigma\in(0,1)$, $\gamma_0>0$, and consider $\rhobar$ from \eqref{1011}.
There exist $\delta_0>0$ such that, if $z(0)\in\cD_{n,\delta_0,\gamma_0}$, then there exists $T_{\mathrm{coll}}^{\partial\Om}>0$ such that the evolution $z(t)$ is defined for $t\in[0,T_{\mathrm{coll}}^{\partial\Om}]$, 
$z(t)\in\Omega^n$ for $t\in[0,T_{\mathrm{coll}}^{\partial\Om})$, and $z_1(T_{\mathrm{coll}}^{\partial\Om})\in\partial\Omega$ and $z'(T_{\mathrm{coll}}^{\partial\Om})\in\Omega^{n-1}$.
Furthermore, as $\delta_0\to0$, the following estimate holds
\begin{equation}\label{eq:ubht}
  T_{\mathrm{coll}}^{\partial\Om}\leq2\pi\delta_0^2+O(\delta_0^3).
  \end{equation}
\end{theorem}
\begin{proof}[Sketch of the proof]
We first find an upper bound on the collision time for dislocation $z_1$ hitting the boundary, conditional on the configuration $z$ remaining in $\cD_{n,\delta,\gamma}$.
Then, after fixing $\gamma_0\in(0,\diam\Omega/2)$, we establish a lower bound on the time at which the configuration $z$ leaves the set $\cD_{n,\delta,\gamma_0/2}$ due to $d_{n-1}(z')$ becoming smaller than $\gamma_0/2$.
The proof is concluded by finding conditions on $\delta$ under which the former collision time is smaller than the latter. 
\end{proof}

We now turn to a scenario for collisions of dislocations.
We will find sufficient conditions for a collision between two dislocations to occur before any other collision event.
We suppose that we have $n\in\N$ ($n\geq2$) dislocations in $\Omega$ two of which, $z_1$ and $z_2$, with Burgers moduli $b_1=+1$ and $b_2=-1$, are much closer to each other than the others; we also suppose that the other $n-2$ dislocations, $z_3,\ldots,z_n$, are sufficiently distant from each other and from the boundary.
We will show that $z_1$ and $z_2$ will collide in finite time, and that this happens before any other collision event occurs.

We define $z'':=(z_3,\ldots,z_n)$ so that a trajectory of the evolution of the configuration of the $n$ dislocations can be represented by the vector $z(t):=(z_1(t),z_2(t),z''(t))\in\Omega^n$.
In this case, the meaningful trajectories for our scenario are those that lie within sets of the form
\begin{equation*}
\begin{split}
  \cC_{n,\zeta,\eta}:=\Big\{(z_1,z_2,z'')\in\Omega^n : &\; |z_1-z_2|< \zeta, d_{n-2}(z'')>\eta, \\ 
  &\; \dist(\{z_1,z_2\},\{z_3,\ldots,z_n\}\cup\partial\Om)>\eta\Big\},
   \end{split}
\end{equation*}
with $\zeta<\eta$ properly quantified (see \cite{HM} for the details).
The geometric meaning of the set $ \mathcal{C}_{n,\zeta,\eta}$ defined above is the following: if $z\in  \mathcal{C}_{n,\zeta,\eta}$, it means that $z_1$ and $z_2$ lie at a distance of at most $\zeta$ from each other, while all the other dislocations $z_3,\ldots,z_n$ lie at a distance of at least $\eta$ away from the boundary and their mutual distance is also at least $\eta$. Moreover, $z_1$ and $z_2$ are at least $\eta$ far away from any other dislocations and from the boundary.

\begin{theorem}[{(c) -- collision between dislocations \cite{HM}}]\label{thm:collision}
Let $n\in\N$ ($n\geq2$) and let $\eta_0\in(0,\diam\Om/2)$.
There exists $\zeta_0>0$ such that, if $z(0)\in\mathcal{C}_{n,\zeta_0,\eta_0}$, then there exists $T_{\mathrm{coll}}^{\pm}>0$ such that the evolution $z(t)$ is defined for $t\in[0,T_{\mathrm{coll}}^{\pm}]$, $z(t)\in\Omega^n$ for $t\in[0,T_{\mathrm{coll}}^{\pm})$, and $z_1(T_{\mathrm{coll}}^{\pm})=z_2(T_{\mathrm{coll}}^{\pm})\in\Omega$ and $z''(T_{\mathrm{coll}}^{\pm})\in\Omega^{n-2}$.
Furthermore, as $\zeta_0\to0$, the following estimate holds
\begin{equation}\label{Tcollprime}
  T_{\mathrm{coll}}^{\pm}\leq\frac{\pi\zeta_0^2\eta_0^2}{2(\eta_0^2-\zeta_0^2-2(n-2)\zeta_0\eta_0)}.
  \end{equation}
\end{theorem}
\begin{proof}[Sketch of the proof]
As for proving Theorem \ref{thm:boundary_collision}, we first find an upper bound on the collision time for dislocations $z_1$ and $z_2$, conditional on the configuration $z$ remaining in $\mathcal{C}_{n,\zeta,\eta}$.
Then,  after fixing $\eta_0\in(0,\diam\Omega/2)$, we establish a lower bound on the time at which the configuration $z$ leaves the set $\mathcal{C}_{n,\zeta,\eta_0/2}$ due to $d_{n-2}(z'')$ becoming smaller than $\eta_0/2$.
The proof is concluded by finding conditions on $\zeta$ under which the former collision time is smaller than the latter. 
\end{proof}

\section{CONFINEMENT UNDER SUITABLE BOUNDARY CONDITIONS}\label{lmsz}
In this section, we consider a geometrical setting which is slightly different from the one in Section \ref{tom}, but more suitable for expressing the coming results. 
We start by assuming that
\begin{equation}\label{H1}
\text{$\Omega$ is a bounded convex open set with $C^1$ boundary,}
\end{equation}
and we consider the lattice spacing of the material to be $2 \pi$ (instead of $1$ as before) and that all the Burgers vectors are oriented in the same direction.
Therefore, any dislocation line is characterised by a Burgers vector of magnitude $2\pi$. 
Moreover, we assume that an \emph{external strain} acts on the crystal: we prescribe the tangential strain on $\partial \Om$ to be a function
\begin{equation}\label{datum}
f\in L^{1}(\partial \Om) \quad \text{with}\quad \int_{\partial \Om} f(x)\,\de x=2\pi.
\end{equation}
This choice of the external strain will determine at most one dislocation inside $\Om$, which we denote by $a$.
Thus, the strain of the deformed crystal is represented by a field
$h_a \in L^1(\Omega;\mathbb R^2)\cap L^2_{\text{loc}}(\Omega\setminus\{a\};\mathbb R^2)$, 
solution to \eqref{100} complemented by the Dirichlet boundary condition
\begin{equation}\label{Fa}
\begin{cases}
\div h_a = 0 & \text{in } \Omega,
\\
\curl h_a = 2\pi \delta_a & \text{in } \Omega,
\\
h_a \cdot \tau= f  &  \text{on }\partial \Omega,
\end{cases}
\end{equation}
where $\tau$ is the tangent unit vector to $\partial \Om$.

By means of the \emph{core radius approach}, one can approximate $h_a$ in $L^2_{\mathrm{loc}}(\Om\setminus\{a\};\R2)$ by a sequence $h_a^\eps\in L^2(\oea;\R2)$ in the perforated domain $\Om_\eps(a):=\Om\setminus\overline B_\eps(a)$, solutions to
\begin{equation}\label{Fa2}
\begin{cases}
\div h_a^\eps = 0  & \text{in } \oea,
\\
\curl h_a^\eps=0 & \text{in } \oea,
\\
h_a^\eps \cdot \tau= f &  \text{on }\partial \Omega\setminus \overline{B}_\eps(a),
\\
h_a^\eps\cdot \nu = 0  &  \text{on }\partial B_\eps(a) \cap \Omega,
\end{cases}
\end{equation}
where $\nu$ is the outer unit normal to $\oea$ (observe that in this case $\nu$ is the inner unit normal to $B_\eps(a)\cap\Omega$), see \cite{BM} and \cite[Theorem 4.1]{CL}.
System \eqref{Fa2} characterises the minimisers of the energy functional (compare with \eqref{104})
\begin{equation}\label{energy1}
\cE_\eps(a):=\min \bigg\{\frac{1}{2}\int_{\oea} \!\!|h|^2\,\de x:  \text{$ h\in L^2(\oea;\mathbb R^2)$, $\curl  h=0$, $ h\cdot \tau=f$ on $\partial \Om\setminus \overline{B}_\eps(a)$}\bigg\}.
\end{equation}
In view of \eqref{103}, we study the asymptotic behaviour of the functionals $\mathcal F_\eps\colon\overline\Om\to\R{}$ defined by
\begin{equation}\label{Feps}
\mathcal F_\eps(a):=\mathcal E_\eps(a)-\pi|\log \eps|
\end{equation}
as $\eps\to0$, to some limit functional $\mathcal F\colon\overline\Omega\to\R{}\cup\{+\infty\}$.
To make the asymptotics precise, we resort to the notion of \emph{continuous convergence} \cite[Definition 4.7]{dalmaso}:
we say that the sequence of functionals $\mathcal F_\eps$ continuously converge in $\overline\Om$ to $\mathcal F$ as $\eps\to0$ if,
for any sequence of points $a^\eps\in\overline\Om$ converging to $a\in\overline\Om$, the sequence (of real numbers) $\mathcal F_\eps(a^\eps)$ converges to $\cF(a)$.
It is well known that the notion of continuous convergence is equivalent to the $\Gamma$-convergence of $\cF_\eps$ and $-\cF_\eps$ to $\cF$ and $-\cF$, respectively, see \cite[Remark 4.9]{dalmaso}. 
Let $d(a):=d_1(a)=\dist(a,\partial\Om)$, let $b\in\partial\Omega\cap\partial B_{d(a)}(a)$, and let $(\rho_a,\theta_a)$ be the polar coordinate system centred at $a$ such that the point $b$ has angular coordinate $\theta_a(b)=0$.
We denote by $\hat\rho_a$ and $\hat\theta_a$ the unit vectors associated with the polar coordinates.
Finally let $g\colon \partial\Om\to\R{}$ be a primitive of $f$ with a jump point at $b$. 
\begin{theorem}[{(d) -- confinement of one dislocation \cite{LMSZ}}]\label{thGc}
Under the assumptions \eqref{H1} and \eqref{datum}, as $\eps\to0$ the functionals $\cF_\eps$ defined by \eqref{Feps} continuously converge in $\overline\Om$ to the functional $\mathcal F:\overline\Om\to\R{}\cup\{+\infty\}$ defined as
\begin{equation}\label{Gammalimit}
\cF(a):=\pi\log d(a)+\frac{1}{2}\int_{\Om_{d(a)}(a)}|K_a + \nabla v_a|^2\, \de x+\frac{1}{2}\int_{B_{d(a)}(a)}|\nabla v_a|^2\,\de x,
\end{equation}
if $a\in\Om$, and $\cF(a):=+\infty$ otherwise. 
Here $K_a(x):=\rho_a^{-1}(x)\hat\theta_a(x)$ and $v_a$ is the solution to
\begin{equation*}
\begin{cases}
\Delta  v_a=0&\text{ in  }\Omega, \\
v_a=g-\theta_a&\text{ on  }\partial\Omega.
\end{cases}  
\end{equation*}
In particular, $\cF$ is continuous over $\overline\Om$ and diverges to $+\infty$ as the dislocation approaches the boundary, that is, $\cF(a)\to +\infty$ as $d(a)\to0$. 
Thus, $\cF$ attains its minimum in the interior of $\Om$.
\end{theorem}
A consequence of Theorem \ref{thGc} is that also the energies \eqref{energy1} attain their minimum in the interior of $\Om$.
\begin{corollary}[{(d) -- confinement of one dislocation \cite{LMSZ}}]\label{confino}
Under the assumptions \eqref{H1} and \eqref{datum}, there exists $\euno>0$ such that, for every $\eps\in(0,\euno)$, the infimum problem
\begin{equation}\label{LMSZ}
\inf\{\cE(a): a \in \overline\Omega\}
\end{equation}
admits a minimiser only in the interior of $\Om$. 
Moreover, if $a^\eps\in \Omega$ is a minimiser for \eqref{LMSZ}, then (up to subsequences) $a^\eps\to a$ and $\cF_a(a^\eps)\to \cF(a)$, as $\eps\to 0$, where $a$ is a minimiser of the functional $\cF$ defined in \eqref{Gammalimit}. 
In particular, for $\eps$ small enough, all the minimisers of problem \eqref{LMSZ} stay uniformly (with respect to $\eps$) far away from the boundary. 
\end{corollary}

having $n\geq2$ dislocations in the crystal is equivalent to imposing a boundary datum with circulation equal to $2\pi n$, which we can obtain by taking the tangential strain $F\cdot\tau$ to be $nf$ on $\partial\Om$, with $f$ as in \eqref{datum}.
Therefore, in analogy with \eqref{energy1}, for $\eps>0$, we define the energy of the $n$-tuple $(a_1,\ldots,a_n)\in \overline\Om{}^n$ as
\begin{equation}\label{energy-n}
\cE_\eps(a_1,\ldots,a_n):=
\min \bigg\{\frac{1}{2}\int_{\Omega_\eps(a_1,\ldots,a_n)} \!\! |h|^2\,\de x\, : \text{$h\in \cX$, $h\cdot \tau=nf$ on $\partial \Om\setminus \bigcup_{i=1}^n \overline{B}_\eps(a_i)$}\bigg\}.
\end{equation}
Here $\Om_\eps(a_1,\ldots,a_n):=\Om\setminus\bigcup_{i=1}^n\overline{B}_\eps(a_i)$ and 
$\cX$ is the space characterized by
\begin{equation}\label{machisei}
h\in \cX:=\cX_\eps(a_1,\ldots,a_n)  \quad \Longleftrightarrow \quad 
\begin{cases}
& h\in L^2(\Omega_\eps(a_1,\ldots,a_n); \R2),
\\
& \curl h = 0 \ \text{in }\mathcal D'(\Omega_\eps(a_1,\ldots,a_n)),
\\
& \int_{\gamma} h\cdot \tau = 2\pi m,
\end{cases}
\end{equation}
where $\gamma$ is an arbitrary simple closed curve in $ \Omega_\eps(a_1,\ldots,a_n)$ winding once counterclockwise around $m$ dislocations. 
In the sequel, for the sake of brevity, we shall omit the dependence on both $\eps$ and the points $a_1,\ldots,a_n$ for the space $\cX$.
Notice that the spaces in \eqref{machisei} are encapsulated, namely, if $0<\eps<\eta$ and if $F\in\cX_\eps(a_1,\ldots,a_n)$, then its restriction to $\Omega_\eta(a_1,\ldots,a_n)$ belongs to $\cX_\eta(a_1,\ldots,a_n)$, since $\Omega_\eta(a_1,\ldots,a_n)\subset\Omega_\eps(a_1,\ldots,a_n)$.
In virtue of this, it follows that
\begin{equation}\label{ginevra}
\cE_\eps(a_1,\ldots,a_n)\geq\cE_\eta(a_1,\ldots,a_n).
\end{equation}
A computation similar to \eqref{103} shows that the energy defined in \eqref{energy-n} behaves asymptotically like $C|\log\eps|$ as $\eps\rightarrow0$, where $C$ depends on the mutual positions of the dislocations. 
In particular, if the $a_i$'s are all distinct and inside $\Om$, the energy diverges like $\pi n|\log\eps|$.
This suggests to study the asymptotic behavior, as $\eps\to0$, of the functionals $\cF_\eps\colon\overline\Om{}^n\to\R{}\cup\{+\infty\}$ defined by
\begin{equation}\label{Fepsn}
\cF_\eps(a_1,\ldots,a_n):= \cE_\eps(a_1,\ldots,a_n)-\pi n|\log\eps|.
\end{equation}
In this context, we say that the sequence of functionals $\cF_\eps$ continuously converge in $\overline\Om{}^n$ to $\cF$ as $\eps\to0$ if, for any sequence of points $(a_1^\eps,\ldots,a_n^\eps)\in\overline\Om{}^n$ converging to $(a_1,\ldots,a_n)\in\overline\Om{}^n$, the sequence (of real numbers) $\cF_\eps(a_1^\eps,\ldots,a_n^\eps)$ converges to $\cF(a_1,\ldots,a_n)$.
Let $g\colon \partial\Om\to\R{}$ be a primitive of $f$ with $n$ jump points $b_i\in\partial \Om$ where the amplitude of each jump is $2\pi/n$, 
and for every $i\in\{ 1,\ldots,n \}$ set
\begin{equation}\label{dconi}
d_i:=\min_{j\in \{1,\ldots,n \}, j\neq i}\bigg\{\frac{|a_i-a_j|}2,\dist(a_i,\partial \Om)\bigg\}.
\end{equation}

\begin{theorem}[{(e) -- confinement of many dislocations \cite{LMSZ}}]\label{thGcn}
Let $n\geq 2$.
Under the assumptions \eqref{H1} and \eqref{datum}, as $\eps\to0$ the functionals $\cF_\eps$ defined by \eqref{Fepsn} continuously converge in $\overline\Om{}^n$ to the functional $\cF:\overline\Om{}^n\to\R{}\cup\{+\infty\}$ defined as
\begin{equation}\label{Gammalimitn}
\begin{split}
\cF(a_1,\ldots,a_n):= \sum_{i=1}^n\pi\log d_i&+\frac{1}{2}\int_\Om |\nabla v_{a_1,\ldots,a_n}|^2\,\de x+\sum_{i=1}^n\frac12\int_{\Om_{d_i}(a_i)} |K_{a_i}|^2\,\de x \\
&+\sum_{i=1}^n\int_{\Om_{d_i}(a_i)} \nabla v_{a_1,\ldots,a_n}\cdot K_{a_i}\,\de x+\sum_{i<j} \int_\Om K_{a_i}\cdot K_{a_j}\,\de x,
\end{split}
\end{equation}
if $(a_1,\ldots,a_n)\in\Om^n$ with $a_i\neq a_j$ for every $i\neq j$, and $\mathcal F(a_1,\dots,a_n):= +\infty$ otherwise. 
Here $K_{a_i}(x):=\rho_{a_i}^{-1}(x)\hat\theta_{a_i}(x)$ and $v_{a_1,\ldots,a_n}$ is the solution to
\begin{equation*}
\begin{cases}
\Delta  v_{a_1,\ldots,a_n}=0&\text{ in  }\Omega, \\
v_{a_1,\ldots,a_n}=ng-\sum_{i=1}^n\theta_{a_i} &\text{ on  }\partial\Omega.
\end{cases}  
\end{equation*}
In particular, $\cF$ is continuous in $\overline\Om{}^n$ and diverges to $+\infty$ if either at least one dislocation approaches the boundary or at least two dislocations collide, that is, $\cF(a_1,\dots,a_n)\to +\infty$ as $d_i\to0$ for some $i$. 
Thus, $\cF$ attains its minimum in the interior of $\overline{\Om}{}^n$, at an $n$-tuple of distinct points.
\end{theorem}

Notice that \eqref{Gammalimitn} can be expressed in term of the functionals $\cF(a_i)$ defined in \eqref{Gammalimit} for each single dislocation $a_i$, namely,
\begin{equation*}
\cF(a_1,\ldots,a_n)=\sum_{i=1}^n\mathcal F(a_i)+\sum_{i<j}\int_\Om(K_{a_i}+ \nabla v_{a_i})\cdot (K_{a_j}+ \nabla v_{a_j})\, \de x.
\end{equation*}
A consequence of Theorem \ref{thGcn} is that also the energies \eqref{energy-n} attain their minimum in the interior of $\Om{}^n$ at an $n$-tuple of well separated points.
\begin{corollary}[{(e) -- confinement of many dislocations \cite{LMSZ}}]\label{confinon}
Let $n\geq 2$. Under the assumptions \eqref{H1} and \eqref{datum}, there exists $\edue>0$ such that, for every $\eps\in(0,\edue)$, the infimum problem
\begin{equation}\label{LMSZ2}
\inf\{\cE_\eps(a_1,\dots,a_n): (a_1,\cdots,a_n) \in \overline\Omega{}^n\}
\end{equation}
admits a minimiser only in the interior of $\overline\Om{}^n$, at an $n$-tuple of distinct points. Moreover, if $(a_1^\eps,\dots,a_n^\eps)\in \Omega^n$ is a minimiser for \eqref{LMSZ2}, then (up to subsequences) we have $(a_1^\eps,\dots,a_n^\eps)\to (a_1,\dots,a_n)$ and $\cF_\eps(a_1^\eps,\dots,a_n^\eps)\to \cF(a_1,\dots,a_n)$, as $\eps\to 0$, where $(a_1,\dots,a_n)$ is a minimiser of the functional $\cF$ defined in \eqref{Gammalimitn}. 
In particular, for $\eps$ small enough, all the minimisers of problem \eqref{LMSZ} are $n$-tuples of distinct points that stay uniformly (with respect to $\eps$) far away from the boundary and from one another. 
\end{corollary}

The proofs of Theorem \ref{thGc}, Corollary \ref{confino}, Theorem \ref{thGcn}, and Corollary \ref{confinon} are extremely technical and beyond the scopes of this note.
To prove Theorem \ref{thGc} some work on harmonic function is needed, as well as the proof that the individual contributions to the energies $\cF_\eps$ convergence to the analogous contribution to the energy $\cF$ defined in \eqref{Gammalimit}.
The proofs of Corollaries \ref{confino} and \ref{confinon} rely on the relationship between $\Gamma$-convergence and continuous convergence.
Finally, the proof of Theorem \ref{thGcn} requires an iterative procedure to treat dislocations that get too close to one another, showing that it is energetically favourable to spread apart.
We refer the interested reader to the manuscript \cite{LMSZ} for the details.

\newpage
\section{CONCLUSIONS}\label{conclusions}
The results presented here concern the dynamics and the equilibrium configuration of screw dislocations un two different contexts, that where the domain has a free boundary, dealt with in Section \ref{tom}, and that where the domain is constrained by a Dirichlet-type boundary condition, dealt with in Section \ref{lmsz}.
Even though both types of results are somewhat known in the literature (especially in the engineering one), we stress that here for the first time quantitative characterisations of the well-known qualitative properties have been provided, and, to the best of our knowledge, this is the first time where such precise estimates were obtained.

Some comments are in order on some choices made in the presentation:
\begin{itemize}
\item in Section \ref{tom}, we dealt with unconstrained dynamics which is the crucial first step towards considering more realistic choices of mobility, such as enforcing \emph{glide directions} \cite{CG,BFLM} or other more general nonlinear \emph{mobilities} \cite{hudson}.
The mobility function usually intervenes in formula \eqref{500}; here we have taken it equal to the identity.
Various suggestions for possible mobility functions can be found in \cite{CG}.
For a specific choice of the mobility, \eqref{500} takes the form of a differential inclusion, and was studied both in \cite{BFLM} to obtain existence and uniqueness results, and in \cite{BvMM15} from the point of view of gradient flows.
\item the choice made in Section \ref{lmsz} to take all the Burgers moduli with the same sign is a simplification introduced in our model, but we expect that the results that we obtain also hold if dislocations with opposite Burgers vectors are allowed.
\end{itemize}

Among the natural steps to follow to continue the research in this direction there are interesting and current research themes such as including thermal effects in the model, or upscaling the number of dislocations $n\to\infty$ to see if plastic behaviours such as that obtained in \cite{GLP} can be recovered.


\begin{thebibliography}{99}
\bibitem{volterra} {V.\@ Volterra}, Sur l'\'equilibre des corps \'elastiques multiplement connexes. \emph{Annales scientifiques de l'\'Ecole Normale Sup\'erieure}, \textbf{24}, 401--517, 1907.

\bibitem{taylor} {G.\@ I.\@ Taylor}, The Mechanism of Plastic Deformation of Crystals. Part I. Theoretical. \emph{Proceedings of the Royal Society of London. Series A}, \textbf{145}, 362--387, 1934.

\bibitem{orowan} {E.\@ Orowan}, Zur kristallplastizit\"at. III. \emph{Zeitschrift f\"ur Physik}, \textbf{89}, 634--659, 1934.

\bibitem{polanyi} {M.\@ Polanyi}, \"Uber eine art gitterst\"orung, die einen kristall plastisch machen k\"onnte. \emph{Zeitschrift f\"ur Physik}, \textbf{89}, 660--664, 1934.

\bibitem{GLP} {A.\@ Garroni, G.\@ Leoni, and M.\@ Ponsiglione}, Gradient theory for plasticity via homogenization of discrete dislocations. \emph{Journal of the European Mathematical Society}, \textbf{12}, 1231--1266, 2010.

\bibitem{HHW} {P.\@ B.\@ Hirsch, R.\@ W.\@ Horne, and M.\@ J.\@ Whelan}, Direct observations of the arrangement and motion of dislocations in aluminium, \emph{Phil. Mag}, \textbf{1}, 677--684, 1956.

\bibitem{nabarro} {F.\@ R.\@ N.\@ Nabarro}, \emph{Theory of crystal dislocations}, International series of monographs on physics, Clarendon, 1967.

\bibitem{HL} {J.\@ P.\@ Hirth and J.\@ Lothe}, \emph{Theory of Dislocations}. Krieger Publishing Company, 1992.

\bibitem{HB} {D.\@ Hull and D.\@ J.\@ Bacon}, \emph{Introduction to dislocations}. Butterworth-Heinemann, 2001.

\bibitem{CG} {P.\@ Cermelli and M.\@ E.\@ Gurtin}, The motion of screw dislocations in crystalline materials undergoing antiplane shear: glide, cross-slip, fine cross-slip. \emph{Archive for Rational Mechanics and Analysis}, \textbf{148}, 3--52, 1999.

\bibitem{BFLM} {T.\@ Blass, I.\@ Fonseca, G.\@ Leoni, and M.\@ Morandotti}, Dynamics for systems of screw dislocations. \emph{SIAM Journal on Applied Mathematics}, \textbf{75}, 393--419, 2015.

\bibitem{BBH} {F.\@ Bethuel, H.\@ Brezis, and F.\@ H\'elein}, \emph{Ginzburg-{L}andau vortices}, Progress in Nonlinear Differential Equations and their Applications, {13}, Birkh\"auser Boston, Inc., Boston, MA, 1994.

\bibitem{SS} {E.\@ Sandier and S.\@ Serfaty}, \emph{Vortices in the magnetic {G}inzburg-{L}andau model}. {Progress in Nonlinear Differential Equations and their  Applications}, {70}, {Birkh\"auser Boston, Inc., Boston, MA}, 2007.

\bibitem{HM} {T.\@ Hudson and M.\@ Morandotti}, Qualitative properties of dislocation dynamics: collisions and boundary behaviour, arXiv:1703.02474, \emph{submitted}.

\bibitem{LMSZ} {I.\@ Lucardesi, M.\@ Morandotti, R.\@ Scala, and D.\@ Zucco}, Confinement of dislocations inside a crystal with a prescribed external strain, arXiv:1610.06852, \emph{submitted}.

\bibitem{H14} {L.\@ L.\@ Helms}, \emph{Potential theory}. {Universitext}, {Springer, London}, {2014}.

\bibitem{CF85} {L.\@ A.\@ Caffarelli and A.\@ Friedman}, {Convexity of solutions of semilinear elliptic equations}. \emph{Duke Mathematical Journal}, \textbf{52}(2), 431--456, 1985.

\bibitem{ADLGP} {R.\@ Alicandro, L.\@ De Luca, A.\@ Garroni, and M.\@ Ponsiglione}, Metastability and Dynamics of Discrete Topological Singularities in Two Dimensions: A  $\Gamma$-Convergence Approach. \emph{Archive for Rational Mechanics and Analysis}, \textbf{214}, 269--330, 2014.

\bibitem{BM} {T.\@ Blass and M.\@ Morandotti}, Renormalized energy and Peach-K\"ohler forces for screw dislocations with antiplane shear. \emph{Journal of Convex Analysis}, \textbf{24}(2), 547--570, 2017.

\bibitem{CL} {P.\@ Cermelli and G.\@ Leoni}, Renormalized energy and forces on dislocations. \emph{SIAM Journal on Mathematical Analysis}, \textbf{37}, 1131--1160 (electronic), 2005.

\bibitem{dalmaso} {G.\@ {Dal Maso}}, \emph{An introduction to {$\Gamma$}-convergence}. Progress in Nonlinear Differential Equations and their Applications, {8}, Birkh{\"a}user Boston Inc., Boston, MA, 1993.

\bibitem{hudson} T.\@ Hudson, Upscaling a model for the thermally-driven motion of screw dislocations. \emph{Archive for Rational Mechanics and Analysis}, \textbf{224}(1), 291--352, 2017.

\bibitem{BvMM15} {G.\@ A.\@ Bonaschi, P.\@ van Meurs, and M.\@ Morandotti}, {Dynamics of screw dislocations: a generalised minimising-movements scheme approach}. \emph{European Journal of Applied Mathematics}, \textbf{28}(4), 636--655, 2017.

























\end{thebibliography}
\end{document}